\RequirePackage{fix-cm}
\documentclass[smallextended]{svjour3}       

\usepackage{amssymb,amsfonts,amsmath}
\usepackage{graphicx}
\usepackage{lineno}
\usepackage{url}
\usepackage{cite}
\usepackage{xcolor}
\usepackage{enumitem}
\usepackage{graphicx}
\usepackage{epstopdf}
\usepackage{soul}
\usepackage{hyperref}
\usepackage{amsmath}
\usepackage{amssymb}
\usepackage{enumitem}
\usepackage{lscape}
\usepackage{multirow}
\usepackage{dsfont}
\usepackage{algorithm,lmodern}
 \usepackage{algpseudocode}
 \definecolor{Color_Pedro}{rgb}{0.0, 0.75, 1.0}

\usepackage[]{varioref, hyperref}

\smartqed  
\usepackage{graphicx}
\journalname{Optimization Letters}
\usepackage{todonotes}

\newcommand{\m}{\mathfrak{m}}
\newcommand{\z}{\mathfrak{z}}
\renewcommand{\L}{\mathfrak{L}}
\newcommand{\C}{\mathcal{C}}

\renewcommand{\P}{\mathbb{P}}
\newcommand{\1}{\mathds{1}}
\newcommand{\R}{\mathbb{R}}
\newcommand{\nuTheta}{\zeta_{\Theta}}

\begin{document}

\title{Differentiability and Approximation of Probability Functions under Gaussian Mixture Models\thanks{The first author was partially supported by ANID-Chile grant  Exploración  13220097. The second and third author were partially supported by Centro de Modelamiento Matem\'{a}tico (CMM), ACE210010 and FB210005, BASAL funds for
 center of excellence, as well as ANID-Chile grants MATH-AMSUD 23-MATH-17, ECOS-ANID ECOS320027, Fondecyt Regular 1220886, Fondecyt Regular 1240120,  and Exploración  13220097. Additionally, the second author was supported by ANID-Chile grants MATH-AMSUD 23-MATH-09 and Fondecyt Regular 1240335.}}
\titlerunning{Probability Functions under Gaussian Mixture Models}

\author{Gonzalo Contador \and Pedro Pérez-Aros \and Emilio Vilches}



\institute{Gonzalo Contador \at
            Universidad Técnica Federico Santa María, 
           Santiago, Chile. 
            \email{gonzalo.contador@usm.cl}  
       \and Pedro P\'{e}rez-Aros  \at
        Departamento de Ingeniería Matemática and Centro de Modelamiento Matemático  (CNRS UMI 2807), Universidad de Chile, Santiago, Chile. 
            \email{pperez@dim.uchile.cl}  
            \and 
            Emilio Vilches \at 
            Instituto de Ciencias de la Ingenier\'ia, Universidad de O'Higgins,  
            Rancagua, Chile. \\
            Centro de Modelamiento Matemático (CNRS UMI 2807), Santiago, Chile. \email{emilio.vilches@uoh.cl} 
}

\date{Received: date / Accepted: date}

\maketitle

\begin{abstract}
In this work, we study probability functions associated with  Gaussian mixture models. Our primary focus is on extending the use of spherical radial decomposition for multivariate Gaussian random vectors to the context of Gaussian mixture models, which are not inherently spherical, but conditionally so. Specifically, the conditional probability distribution, given a random parameter of the random vector, follows a Gaussian distribution, which allows us to rewrite the probability function as a tractable integrated Gaussian mixture. This assumption, together with spherical radial decomposition for Gaussian random vectors, enables us to represent the probability function as an integral over the Euclidean sphere. Using this representation, we establish sufficient conditions to ensure the differentiability of the probability function and provide an integral representation of its gradient. Furthermore, we approximate the probability function using random sampling over the parameter space and the Euclidean sphere. Finally, we present a numerical example that illustrates the advantages of this approach over classical approximations based on random vector sampling.

 \keywords{stochastic
programming \and chance constrained optimization \and probability functions \and Gaussian mixture models}
\subclass{Primary: 90C15, 65K10 \and Secondary: 62F15}
\end{abstract}

\section{Introduction}
 A chance-constrained optimization problem is a mathematical program of the form:
 \begin{equation}\label{CCO}
		\min  f(x) \text{  s.t.  } \mathbb{P}\left(  g(x,\xi)\leq 0 \right) \geq p, 
\end{equation}
where $f\colon\mathbb{R}^n \rightarrow \mathbb{R}$ is the objective function, $\xi$  is an $m$-dimensional random vector defined on a probability space, $g\colon\mathbb{R}^n \times \mathbb{R}^m \rightarrow \mathbb{R} $ represents an inequality constraint,  and $p \in (0,1)$ is a reliability parameter. In this context, a vector $x \in \mathbb{R}^n$ is feasible for the optimization problem \eqref{CCO} if and only if the random inequality   $g(x,\xi) \leq 0$ is satisfied with probability of at least $p$.   For classical monographs on optimization problems with probabilistic constraints, we refer to \cite{Stprog,Prekopa_1995}.

To efficiently compute numerical solutions for chance-constrained optimization problems, access to both the values and gradients of the probability function, denoted as $\Phi(x) := \mathbb{P}\left(  g(x,\xi)\leq 0 \right)$, is crucial. A typical procedure for computing this function involves sampling the random vector $\xi$ using standard Monte Carlo techniques and then approximating the actual probability value through a sample average. 

The \emph{Law of Large Numbers} guarantees that this approach converges towards the real probability as the number of sampled values of $\xi$ approaches infinity. However, this approach presents notable challenges, especially when the random inequality $g(x,\xi) \leq 0$ has a nonlinear structure with respect to the variable $\xi$. Such a sampling procedure introduces inaccuracies in the probability computations, complicating the solution process for optimization problem.  This complexity arises from the need not only to compute probability values but also to determine gradients of probability functions across multiple iterations of various optimization methods, ultimately making this approach impractical for addressing problem \eqref{CCO}.

Hence, there is a clear need for an efficient procedure to compute both the values and gradients of probability functions. In this context, some authors have utilized the so-called \emph{spherical radial decomposition} as an effective method to calculate the values of the probability function. This decomposition has proven to be  particularly promising for studying first-order variational information of the probability function. Specifically, it provides insights into gradients when the probability function is smooth and subgradients in the non-differentiable case. Most works related to this approach involve the analysis of \emph{elliptical symmetrical distributions}, such as multivariate Gaussian distributions (see, e.g., \cite{MR4000225,MR3881946,MR3273343,vanAckooij2022,vanAckooij_Henrion_2016} and the references therein). 

Nevertheless, there are various random phenomena for which elliptical symmetric distributions are not an adequate model. Our motivation here is to provide a framework for solving chance-constrained problems that do not exhibit an inherently spherical structure as described above for $\xi$ but do so conditionally, based on the realization of a latent variable, as in the case of Gaussian mixture models. This framework leverages known results from the spherical case and extends them to the conditionally Gaussian scenario.

Gaussian mixture models are probabilistic models used to represent normally distributed subpopulations within an overall population \cite{mclachlan2004finite}. A Gaussian mixture model with $K$ components is parameterized by a distribution of weights of non-negative mixture components $c$, the means and covariance matrix of each component, with the density given by $h(\xi)=\sum_{i =1}^K c_i h_i(\xi)$. Here, the weights satisfy $\sum_{i =1}^K c_i=1$ and can be viewed as either learned or as a mixing distribution over components, where  $c_i=\P (\xi \text{ comes from component }i)$ \cite{reynolds2009gaussian}. The function $h_i$ represents the density of a Gaussian vector with mean $\mu_i$ and covariance matrix $\Sigma_i$. Extensions of this model to a combination of uncountably many Gaussian distributions, weighted by a probability density, are well-established in \cite{infinitegaussian}.
 
 The introduction of a mixture framework to account for uncertainty in $\mathbb{P}$, by treating it as a mixture of distributions,  according to a mixing distribution $\m$ over the set of probability measures $\mathcal{P}$, which also allows practitioners to incorporate their beliefs about the overall random structure of the problem, is a recent advancement in chance-constrained optimization. Using standard Bayesian analysis tools, authors have employed specific conjugate structures to derive closed-form posterior constraint distributions in optimization problems related to staffing \cite{aktekin2016stochastic}, hydraulic engineering \cite{nima2015hydraulic}, and portfolio allocation \cite{Lai2011finance}. More recently, approximations to posterior distributions have been developed and shown to be asymptotically optimal and consistent in a statistical sense \cite{jaiswal2023Bayesian}. Moreover, recent contributions by  \cite{pmlr-v108-kirschner20a} and \cite{shapiro2023bayesian}  have even developed the so-called Bayesian Optimization framework,  using observations of random phenomena $\xi_1, \ldots, \xi_n$ to approximate the true constraint distribution. 
 
 In most cases, however, the decision $x$ (a ‘here-and now-decision’) must be made before the random phenomena are observed. For example, in \cite{GonzalezDynamic}, the design of a mechanical structure (encoded by $x$), which is constructed once and permanently, must withstand future random forces $\xi$ with high probability.

  Our proposal leverages a mixture formulation of the constraint probability to derive a tractable representation of the probability function. This approach transforms the original probability function into an expectation over radial conditional expectation functions. Our method facilitates the study of the differentiability of the probability function by combining classical calculus tools with the spherical radial decomposition technique. Additionally, it improves the estimation of the true probability function, reducing its variance when compared to a standard Monte Carlo approach, and enables estimation of its gradient through stable empirical approximations  based on sampling from the parameter set and the unit sphere (spherical radial decomposition).

The paper is organized as follows. In Section \ref{Sect_Notation}, we introduce the main notation used throughout the work and provide the mathematical formulation of our approach. In Section \ref{Sect_Gradient_ofProbability}, we examine the differentiability of the probability function under Gaussian mixture models and present an integral representation using the spherical radial decomposition (see Theorem \ref{thm:gradient}). Section \ref{SectAprx_prob_MC} is dedicated to formally establishing that an approximation based on an i.i.d. sample over the set of parameters and the $m$-dimensional unit sphere, with ad hoc probability distributions, converges uniformly over compact sets under mild assumptions, both in terms of function values and gradient information (see Theorem \ref{thm:limitecontinuo}). In Section \ref{Section_Nume}, we illustrate our results with a numerical example. For clarity, the proofs of the main results are presented in the appendix.
 
\section{Notation and Preliminary}\label{Sect_Notation}
In this paper, we adopt the standard notation commonly used in variational analysis, optimization, statistics, and probability theory. For more details and comprehensive coverage, we refer the reader to the monographs \cite{Rockafellar_Wets_2009,MR3823783,Stprog,Prekopa_1995,Bill86}.

	\subsection{Normal Distribution and Spherical Radial Decomposition}\label{SECTION_SRD_GAUSSIAN}
	
	It is known that if a random vector $\xi \in \mathbb{R}^m$ has a (multivariate) normal distribution $\xi \sim \mathcal{N}(\mu,\Sigma)$, then $\xi$ admits a so-called \emph{spherical radial decomposition} representation, given by
	\begin{equation*}
		\xi=\mu+RL\Theta,
	\end{equation*}
	where $R$ follows a $\chi$-distribution with $m$-degrees of freedom, $\Theta$ follows a uniform distribution $\nuTheta$ over the $m$-dimensional unit sphere $\mathbb{S}^{m-1}:=\{ z\in \mathbb{R}^m :
	\sum_{i=1}^m z_i^2=1 \}$, and $L$ is the (unique) Cholesky decomposition of $\Sigma$, such that $LL^T=\Sigma$.  

 Now, given a probability function $\varphi(x):= \mathbb{P}( g(x,\xi) \leq 0)$, for $\xi \sim \mathcal{N}(\mu,\Sigma)$,  the above decomposition allows us to rewrite $\varphi$ as \begin{equation*}
		\varphi(x) = \int_{\mathbb{S}^{m-1}}  	e(x,v)  d\nuTheta(v),
	\end{equation*}
	where $\nuTheta$ is the uniform measure on $\mathbb{S}^{m-1}$, and $e(x,v)$ represents the probability of satisfying the constraint for a given $x$ when $\Theta=v$ is fixed. Specifically,  
	\begin{equation}
		e(x,v) = \int_{\{r \geq 0 \; : g(x,\mu + rLv) \leq 0\} } \chi (r)dr,
		\label{kernel}
	\end{equation}
	where  $\chi$ is the density of a $\chi$-distribution with $m$-degrees of freedom, defined by
	\begin{align}\label{density-like} 
		\chi(r):=   \frac{1
		}{2^{m/2-1}\Gamma(\frac{m}{2})}  r^{m-1}\exp\left(-\frac{r^2}{2}\right) \quad \textrm{ for } r\geq 0.
	\end{align}
For further details, see \cite{MR3273343,vanAckooij_Henrion_2016} and the references therein.
  
 \subsection{Gaussian Mixture Models}

 A finite Gaussian mixture model in $\mathbb{R}^m$ can be written as a density $p$ of the form 
 \begin{equation*}
     \xi \sim p(\,\cdot\,|\mu_1, \ldots ,\mu_k, \Sigma_1, \ldots, \Sigma_k, w_1, \ldots , w_k)=\sum_{i=1}^k w_i\mathcal{N}(\mu_i,\Sigma_i),
 \end{equation*}
where $k\in \mathbb{N}$ is a fixed integer, $\mu_1, \ldots ,\mu_k$ are the means and $\Sigma_1, \ldots, \Sigma_k$ are the covariances of $k$ Gaussian distributions, and $w_1, \ldots, w_k$ are the mixing proportions, which are positive and sum to 1. In this way, the model can be interpreted as $\xi$ having a multivariate normal distribution with parameters $\mu_k$ and $\Sigma_k$ with probability $w_k$. 

To extend the mixture model beyond a finite number of components,  the authors in \cite{infinitegaussian} assume that vector $\mu$ and positive definite matrix $\Sigma$ are themselves random elements following specific choices of continuous distributions, so that the vector $\xi \sim \mathcal{N}(\mu,\Sigma)$ has parameters that can be chosen among an uncountably infinite set of possibilities. Subsequent work also incorporates uncertainty quantification on the number of categories $K$ \cite{matza2021infinite}, allowing the number of mixture components to be countably infinite. We will generalize the Gaussian Mixture Model so that $\mu$ and $\Sigma$ are selected according to a density $\m$ over a subset of feasible means and covariance matrices, specifically, 

 \begin{equation*}
     \xi|\mu, \Sigma \sim \mathcal{N}(\mu,\Sigma) \text{ and } (\mu, \Sigma) \sim \m.
 \end{equation*}
 This formulation will be detailed in Section \ref{sec:theory}. 
 
	\subsection{Mathematical Formulation and Spherical Radial decomposition for    Gaussian Mixture Models}\label{sec:theory}
  Let us consider  a probability space $( \Omega, \mathcal{S},\mathbb{P})$, a random vector $(\xi, C)\colon \Omega \to \Xi \times \C \subset \mathbb{R}^{m+p}$, and a set of conditional probability  measures (uncertainty set) $\mathcal{P}:=\{ \mathbb{P}_c (A) = \mathbb{E}( 1_{A}|C =c ) : c\in \C  \}$ over $\Xi$, where the parameter space $\C\subset \R^p$ is a Borel measurable set. We denote by $\m$ the measure induced by the second coordinate of $(\xi, C)$ over  the Borel $\sigma$-algebra  $\mathcal{B}(\C)$, that is, 
	\begin{equation}
		\label{eq:prior}
		\m(A)=\mathbb{P}\left( \omega \in \Omega : C(\omega) \in A \right).
	\end{equation}
 This mathematical framework is interpreted as follows: the random element $\xi$ in the chance constraint has a distribution parameterized by $c \in \C$, and $\m$ is a  mixing distribution on $\C$. Therefore, the chance constraint in \eqref{CCO} can be written as
	\begin{equation*}
		\mathbb{P}(g(x,\xi) \leq 0) =\int_\mathcal{C} \mathbb{P}_c(g(x,\xi) \leq 0)d\m(c).
	\end{equation*}
	Hence, problem \eqref{CCO} can be rewritten as:
	\begin{align*}
		\min  f(x) \text{ s.t. } \mathbb{P} (g(x,\xi))\leq 0 )=\int_\mathcal{C}\mathbb{P}_c(g(x,\xi) \leq 0)d\m(c)  \geq p.
	\end{align*}
 We observe that the above formulation coincides with \eqref{CCO} when the distribution $\mathbb{P}$ of $\xi$ is fully specified, i.e., when the uncertainty set $\mathcal{P}$ reduces to a single element $\mathbb{P}$ (and, consequently, $\m=\delta_\mathbb{P}$). In the context of an infinite Gaussian mixture model, we consider a probability space $(\C, \mathcal{F},\m)$ where the random variable $C$ has probability distribution $\m$ as defined in \eqref{eq:prior}, which acts like a mixing distribution over the components of the Gaussian mixture, each of them characterized by a mean and a covariance matrix, as in \cite{infinitegaussian}. Therefore, to describe the Gaussian components, we consider measurable functions $\z \colon \C \to \mathbb{R}^m$ and $  \L \colon \C \to \mathbb{R}^{m\times m}$ such that
	\begin{enumerate}[label=\alph*)]
		\item For each $c\in \C$, conditional on the event $C=c$,  the vector $\xi \sim \mathcal{N}(\z(c),\Sigma(c))$, i.e., 

		\begin{align*}
			\mathbb{P}_c( \xi \in A):=\mathbb{P}( \xi \in A | C=c)=   \int_A h_c(z) dz \textrm{ for all Borel set } A\subset \mathbb{R}^m,
		\end{align*}
  where \begin{align*}
		h_c(z):=	\frac{1}{ \sqrt{(2\pi)^m\det \Sigma(c) }}  \exp\left( - \frac{1}{2} (z- \z(c))^{\top} \Sigma^{-1}(c) (z- \z(c))\right).
		\end{align*}
  We refer to \cite{MR4000225,MR3881946} for more details. 
  \item There exists $ \eta_0>0$ such that
  \begin{equation}\label{eqeta0}
      \|\z(c)\|_2 + \|\L(c)\|_M \leq \eta_0 \text{  for all }c\in \C, 
  \end{equation}
  where $ \L(c)$ is the Cholesky decomposition of the positive definite matrix $\Sigma(c)$, i.e., $\Sigma(c):= \L(c) ^\top \L(c)$, $\|\cdot\|_2$ denotes the Euclidean norm on $\mathbb{R}^m$ and $\|\cdot\|_M$ denotes the matrix norm on the space of $m \times m$ matrices, i.e., $\|A\|_M=\sup\{\|Ax\|_2: \|x\|_2=1\}$.
	\end{enumerate}
	Hence, the expression $h_c(\xi)$ can be understood as the conditional density of $\xi$ given $C=c$, which allows us to define the conditional probability function as 
	\begin{align}\label{eq:condProb}
		\varphi_c(x):=\P_c(g(x,\xi)\leq 0)=\mathbb{P}(g(x,\xi)\leq 0|C=c)= \int_{ \{ z: g(x,z) \leq 0   \} } h_c( z)dz.
	\end{align}
      Condition a) guarantees that $\varphi_c$ is well defined for almost every $c \in \C$ (see, e.g., \cite{Faden1985Conditional}), and from \eqref{eq:condProb}, it follows that  
	\begin{align*}
		\P(g(x,z)\leq 0)&=\int_\C \P_c(g(x,\xi)\leq 0)d\m(c)= \mathbb{E}(\varphi_C(x)).
	\end{align*} 	
 It should be emphasized that the probability defined on \eqref{eq:condProb} can be interpreted as the probability of satisfying the constraint, calculated after randomly selecting a Gaussian distribution according to $\m$.

	Applying the \emph{spherical radial decomposition} on the event $C=c$ for a Gaussian random variable $\xi$, as described in Section \ref{SECTION_SRD_GAUSSIAN}, we obtain the following parameterized decomposition: 
	\begin{align*}
		\xi|\{C=c\}= \z(c) +R\L(c) \Theta.
	\end{align*}
	Moreover, it is clear that $(\Theta, R)$ is independent with $\Theta$ uniform over the sphere $\mathbb{S}^{m-1}$ and  $R$ follows a $\chi$-distribution with $m$ degrees of freedom. 
 
	 The conditional probability function \eqref{eq:condProb} can be expressed as 
	\begin{align*}
		\varphi_c(x) = \int_{\mathbb{S}^{m-1}}  e_c( v, x) d\nuTheta(v),
	\end{align*}
	where $e_c(\cdot,\cdot) := e(c,\cdot,\cdot)$ and $e \colon \C \times \mathbb{S}^{m-1}\times  \mathbb{R}^n  \to \mathbb{R}\cup
	\{ +\infty\}$ is the  parametric \emph{radial probability} function in \eqref{kernel} on the event $C=c$ given by
	\begin{align}\label{function_e}
		e_c(x,v) = 
  \displaystyle  \int\limits_{  \{ r\geq 0 :
			g(x, r\L(c)v+ \z(c))     \leq 0   \}      }   \chi( r)
		dr.
	\end{align}
In what follows, we say that $\bar{x}$ is a $\z$-\emph{uniform Slater point} if there exists an open neighborhood $U$ of $\bar{x}$  and  $\gamma_0 >0$ such that 
	\begin{equation}
		\label{assumption1}
			g(x,\z(c)) < -\gamma_0,  \text{ for all }x\in U \text{ and }\m \text{-almost all }c\in \C.
	\end{equation}

\section{Gradient of Probability Functions under Gaussian Mixture Models}\label{Sect_Gradient_ofProbability}
In this section, we study the following (unconditional) probability function: 
\begin{equation}\label{exp:Prob}
	\Phi(x):=\mathbb{E}(\varphi_C(x))=\int_{ \C }\varphi_c(x)d\m(c),
\end{equation}
where $\varphi_c(x)$ is defined on  \eqref{eq:condProb}. 

The following result, proven in the appendix, establishes that the above probability function is well-defined.
\begin{proposition} \label{prop:econtinua}
Assume that $\z$ and $\Sigma$ are measurable and that $g\colon \mathbb{R}^n \times \mathbb{R}^m \to \mathbb{R}$ is measurable. Then, the parametric radial probability function $e \colon \C \times \mathbb{S}^{m-1}\times \mathbb{R}^n \to \mathbb{R}\cup
	\{ +\infty\}$, defined in \eqref{function_e}, is $\mathcal{F}\otimes \mathcal{B}(\mathbb{S}^{m-1}) \otimes \mathcal{B}(\mathbb{R}^n)$ measurable. Furthermore, for any $v \in \mathbb{S}^{m-1}$ and $\m$-almost every $c \in \C$, the function $e_c(\cdot,v)$ is continuous at any $\z$-uniform Slater point \eqref{assumption1}.
\end{proposition}

Before establishing the differentiability of the probability function defined in   \eqref{exp:Prob}, we need to introduce a condition that allows us to control the gradient of $g$ locally at $\bar x$, but globally on $(z, c) \in \mathbb{R}^m \times  \C $. In our GMM setting, we consider the so-called exponential growth condition, first introduced in \cite{MR3273343}. Without this condition,  differentiability may fail even for  Gaussian distributions  (see, e.g.,  \cite{MR3273343,zbMATH07047588}).

\begin{definition}[Exponential growth condition]\label{def:growthcon}
 We say that a function $g$ satisfies the exponential growth condition at $\bar{x}$ if one can find constants  $\varepsilon,a,b>0$ such that 
	\begin{equation*}
		\| \nabla_x g(x,z ) \| \leq   a\exp(b\| z\|  )  \text{ for all } (x,z) \in 
		\mathbb{B}_{\varepsilon } (\bar{x}) \times \mathbb{R}^m.
	\end{equation*}
\end{definition}
The following result proves the continuous differentiability of the probability function \eqref{exp:Prob} and provides an integral formula for its gradient.   
\begin{theorem} \label{thm:gradient}
	Let $g\colon \mathbb{R}^n\times \mathbb{R}^m \to \R$ be a continuously differentiable function satisfying the exponential growth condition at $\bar{x}$. Assume that $g$ is convex in the second argument and $\bar{x}$ is a $\z$-uniform Slater point as in \eqref{assumption1}. Then, there exists a neighborhood $U'$ of  $\bar{x}$ such that the probability function $\Phi$ defined in \eqref{exp:Prob} is continuously differentiable on $U'$, and 
	\begin{align}\label{Formula_gradient_PHI}
		\nabla \Phi(x) = \int_{   \mathbb{S}^{m-1} \times \C } \nabla_x e_c (x,v) d\nuTheta(v)  d\m (c), \quad  \text{ for all } x\in U'.
	\end{align}
\end{theorem}
The proof of the above result will follow directly from a more general result presented in the appendix, which provides an explicit representation for computing the integral formula in \eqref{Formula_gradient_PHI} in terms of the function $g$, along with a detailed technical analysis of the estimation process for the integrand (see Theorem \ref{theorem01}). We refer to \eqref{grad_form} for an explicit formula for $\nabla_x e_c(x,v)$ in terms of the shape of $g$ and the distribution of $\xi$.

\section{Approximation of Probability Functions via Monte Carlo Simulation} \label{SectAprx_prob_MC} 

In this section, we propose a sampling method to approximate the probability function $\Phi$ in \eqref{exp:Prob} and its gradient. As Theorem \ref{thm:gradient} guarantees the continuity of $\Phi$, our goal is to obtain an expression that remains continuous in $x$. Therefore, we avoid relying on naive i.i.d.-based Monte Carlo approximations of the form
	\begin{equation}
		\label{eq:sampleaverage}
		\tilde{\Phi}_N(x)=\frac{1}{N}\sum_{i=1}^N \1_{ \{ g(x,\xi_i)\leq 0 \}}.
	\end{equation}
 Instead, we propose a Monte Carlo approximation to estimate $\Phi(x)$ and $\nabla \Phi(x)$ that relies on an i.i.d. sample of $\{(v_i,c_i)\}_{i \in \mathbb{N}}$, which can be obtained in several standard ways (for instance, inverse transform sampling \cite{david1948probability} or the use of an asymptotic surrogate Markov Chain Monte Carlo (MCMC) method \cite{mengersen1999mcmc,Jones2004MCCLT} which, after a sufficiently large number of steps, behaves like an i.i.d. sample \cite[Theorem 7.4]{robert_MCMC}. See the above references for details on the speed/rates of convergence and the numerical example in section \ref{Section_Nume} for a detailed sampling scheme). By averaging the radial probability function \eqref{function_e} across the sample, we define the empirical spherical approximation to \eqref{eq:condProb} and its gradient \eqref{Formula_gradient_PHI} as
	\begin{equation}
		\label{eq:MCMCprob}
		\Phi_N(x)=\frac{1}{N}\sum_{i=1}^N e_{c_i}(x,v_i) \quad \text{ and } \quad \nabla\Phi_N(x)=\frac{1}{N}\sum_{i=1}^N \nabla_x e_{c_i}(x,v_i).
	\end{equation}
Hence, under mild regularity and growth conditions, the hypotheses in \cite[Theorem 9.60]{Stprog} are satisfied, leading to the following convergence theorem, which is proved in the appendix.
 	\begin{theorem}\label{thm:limitecontinuo} 
    Let $\{(v_i,c_i)\}_{i=1}^N$ be an i.i.d. sequence obtained from $\nuTheta \otimes \m$.  
    Consider a compact set $X \subset \R^n$ such that, for every $x\in X$, condition \eqref{assumption1} and the exponential growth condition hold at $x$. Then, with probability 1, the approximations $\Phi_N(x)$  and  $\nabla\Phi_N(x)$ defined in \eqref{eq:MCMCprob} converge uniformly on $X$ to $\Phi$ and $\nabla \Phi$, respectively, as $N\to \infty$. 
 \end{theorem}
Theorem \ref{thm:limitecontinuo} enables us to establish asymptotic distributional results to assess the consistency of the approximations in \eqref{eq:MCMCprob} and to establish the asymptotic normality of their errors with controlled variance:
\begin{corollary} \label{prop:erroresnormales} 
Under the assumptions of Theorem \ref{thm:limitecontinuo}, for all $x\in X$, the following holds: 
\begin{equation*}
\sqrt{N}[\Phi_N(x)-\Phi(x)]\to_d \mathcal{N}(0,\sigma^2_x ) \text{ and } \sqrt{N}[\nabla\Phi_N(x)-\nabla\Phi(x)]\to_d \mathcal{N}(0,\Sigma^2_x) \textrm{ as } N\to +\infty,
\end{equation*}
 where 
 \begin{align*}
     \sigma^2_{x}&:= \int_{ \mathbb{S}^{m-1} \times 
     \C}e_c(x,v)^2  d \nuTheta \otimes \m (v,c) - \Phi(x)^2,\\
     \Sigma^2_x &:= \int_{\mathbb{S}^{m-1} \times 
     \C }[\nabla_x e_c(x,v)][\nabla_x e_c(x,v)]^T  d \nuTheta \otimes \m (v,c) - \nabla\Phi(x)\nabla\Phi(x)^T.
 \end{align*}

\end{corollary}  

\begin{proof} Fix $x\in X$. Let $\{(v_i,c_i)\}_{i=1}^N$ be an i.i.d. sequence of $\nuTheta \otimes \m$. 
Then, the sequence $e_{c_i}(x,v_i)$ is also also i.i.d. with mean $\mathbb{E}(e_{c_i}(x,v_i))=\int_\C\int_{\mathbb{S}^{m-1}}  e_{c_i}( v_i, x) d\nuTheta(v_i)d\m(c_i)=\Phi(x)$ and variance $\mathbb{E}(e_{c_i}(x,v_i)^2)-\mathbb{E}(e_{c_i}(x,v_i))^2=\sigma^2_x$. An analogous argument shows that  the sequence $\nabla_xe_{c_i}(x,v_i)$ is i.i.d. with mean $\nabla \Phi(x)$ and covariance matrix $\Sigma^2_x$. Therefore, the result follows from the Central Limit Theorem \cite[Theorem 27.2]{Bill86}.
\end{proof}
The following result establishes convergence in probability for our approximations.
\begin{corollary}
Under the assumptions of Theorem \ref{thm:limitecontinuo}, for all $x\in X$, the following holds:
$$
\Phi_N(x) \to_\mathbb{P} \Phi(x) \textrm{ and } \nabla\Phi_N(x) \to_\mathbb{P} \nabla\Phi(x) \textrm{ as } N\to \infty.
$$
\end{corollary}
\begin{proof} Due to Corollary \ref{prop:erroresnormales}, for all $\varepsilon>0$, we have that $$\mathbb{P}(|\Phi_N(x) - \Phi(x)|>\varepsilon)= O(N^{-1/2})=o(1).$$ The proof is similar for the case of the gradient.
\end{proof}
The naive approximation $\tilde{\Phi}_N$ in \eqref{eq:sampleaverage} is also asymptotically normal, with $\sqrt{N}(\tilde{\Phi}_N(x)-\Phi(x)) \to_d \mathcal{N}(0,\tilde{\sigma}_{x}^2)$ for every $x\in X$ (thanks to the central limit theorem), where the asymptotic variance is given by $\tilde{\sigma}_{x}^2:=\Phi(x)-\Phi(x)^2$. Since for any $(c,v)$, we have $e_c(x,v)\geq e_c(x,v)^2$, it follows that  $\Phi(x)\geq \int_{  \  \mathbb{S}^{m-1} \times \C  }e_c(x,v)^2  d \nuTheta \otimes \m (v,c)$, and therefore $\tilde{\sigma}_{x}^2\geq {\sigma}_{x}^2$ for every $x \in X$. Therefore, as both convergence rates are of order $\sqrt{N}$, the asymptotic relative efficiency of $\Phi_N(x)$ with respect to $\tilde{\Phi}_N(x)$ (this is, the ratio of their asymptotic variances) is no larger than one for every $x\in X$, so estimating the probability function by $\Phi_N$ is uniformly more efficient than doing so via $\tilde{\Phi}_N$.

\section{Estimation of the Probability Function and Its Gradient: A Numerical Example}\label{Section_Nume}
In this section, we illustrate the quality of the proposed approximations. 

Let $C\in [-1,1]$ be a random variable such that $D:=\frac{1}{2}(C+1)$ follows, for some $\delta>0$, a Beta($\delta,\delta$) distribution, that is, 
$$
\m(c):=\frac{\Gamma(2\delta)}{2^{2\delta -1}\Gamma(\delta)^2}(1-c^2)^{\delta-1} \textrm{ for } c\in [-1,1],
$$
where $\Gamma(\cdot)$ denotes the Gamma function. Let $\xi \sim \mathcal{N}\left(\begin{pmatrix}
	0 \\ C
\end{pmatrix}, \begin{bmatrix}
1-C^2 & 0 \\ 0& 1-C^2
\end{bmatrix}\right) $, this is, conditional on $C=c$, the vector $\xi$ has two independent components $\xi_1\sim \mathcal{N}(0, 1-c^2)$ and $\xi_2\sim \mathcal{N}(c, 1-c^2)$ (but marginally they are dependent through the shared C) and suppose that $g\colon\mathbb{R}^n \times \mathbb{R}^2\to \mathbb{R}$ is given by 
$$g(x,\xi)= \Vert \xi\Vert^2 - \Vert x\Vert^2 -2.$$  
The (unconditional) distribution of $\xi$ is given by 

$$f_\xi(z_1,z_2)=\frac{\Gamma(2\delta)}{2^{2\delta}\pi\Gamma(\delta)^2}\int_{-1}^1 (1-c^2)^{\delta-2} \exp\left(-\frac{z_1^2+(z_2-c)^2}{2(1-c^2)}dc\right).$$

For $\delta < 1$, $f_\xi$ has a local minimum at $(0,0)$ and two modal points at $(0,\pm1)$, so we choose values of $\delta \in  (0,1)$ to make the mixture model multimodal. Figure \ref{fig:joint} illustrates the distribution of $\xi$ for the case $\delta=0.5$, highlighting its clear two modal points.

\begin{figure}[!ht]
    \centering
    \includegraphics[width=.495\linewidth]{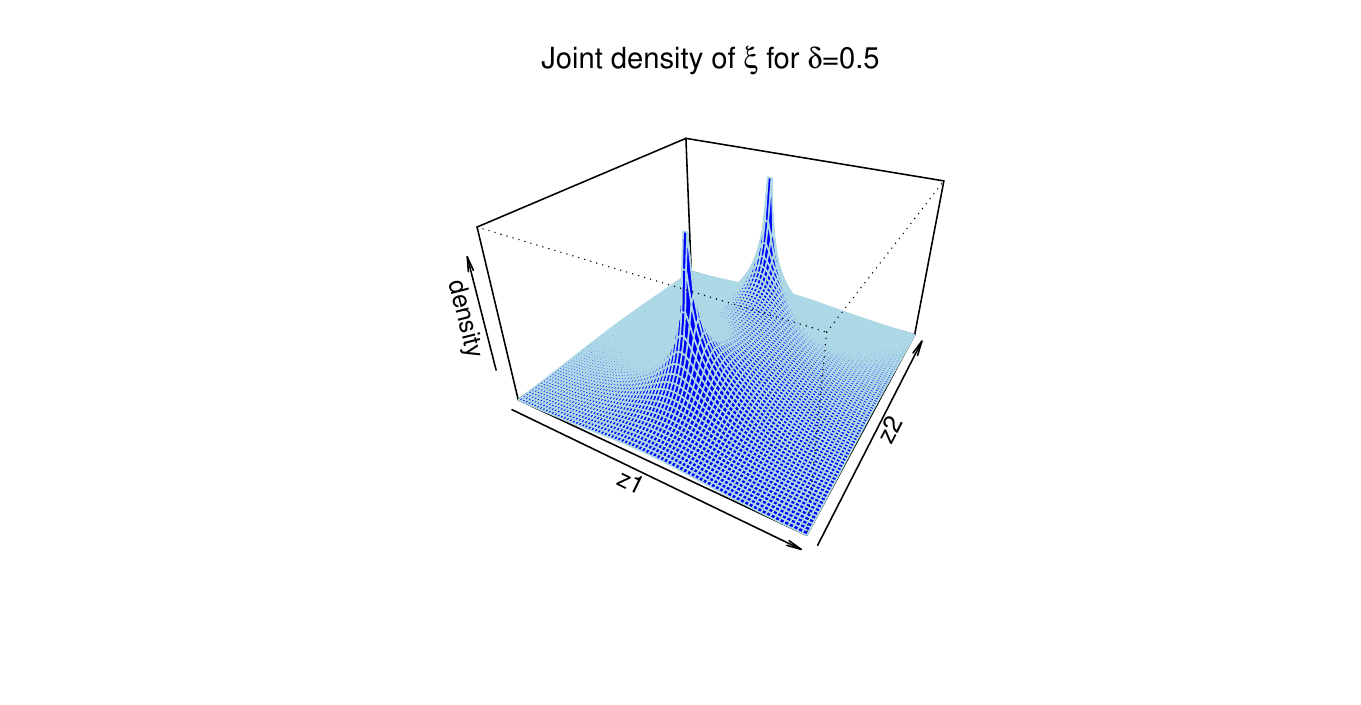}
    \includegraphics[width=.495\linewidth]{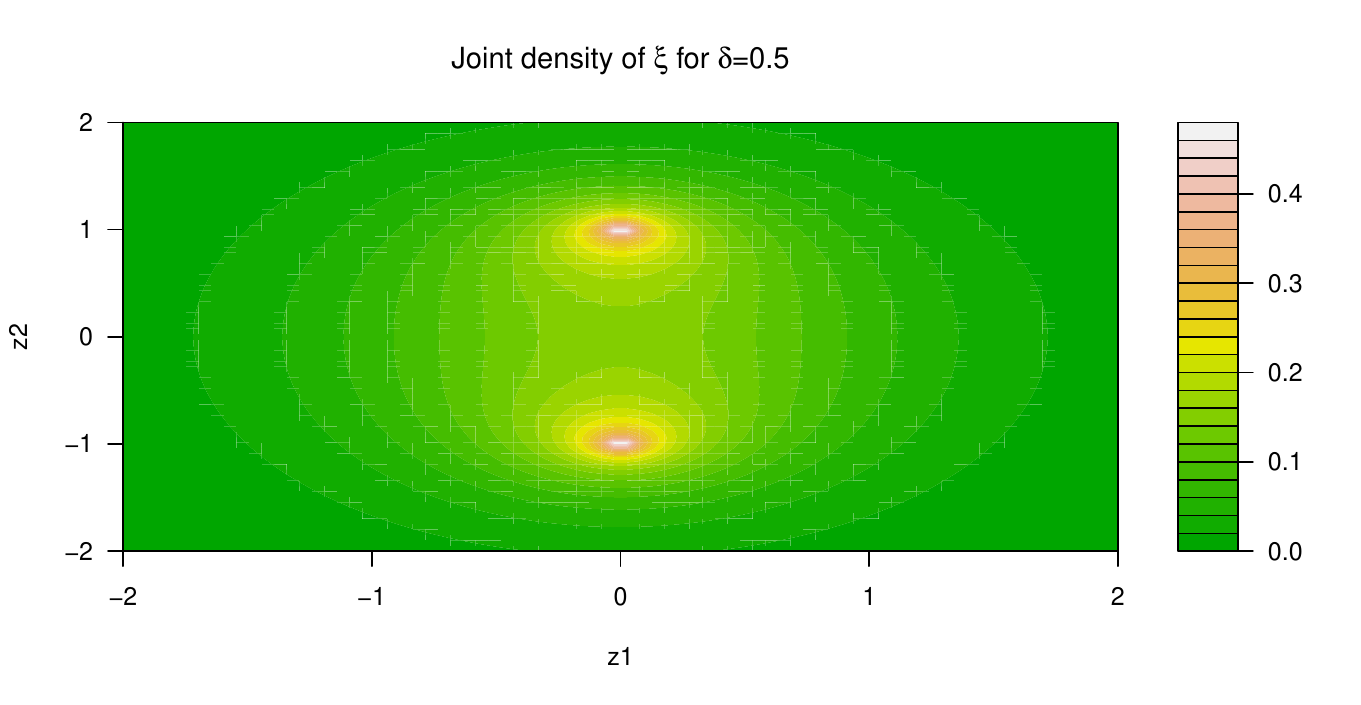}
    \caption{3-D plot (left) and heat map (right) of the joint density of $\xi$ for $\delta=0.5$.}
    \label{fig:joint}
\end{figure}

In this setting, the probability \eqref{exp:Prob} becomes
\begin{align*}
  \Phi(x)&=\int_{-1}^1\mathbb{P}(g(x,\xi) \leq 0  |C=c)d\m(c)=\frac{\Gamma(2\delta)}{2^{2\delta-1}\Gamma(\delta)^2}\int_{-1}^1(1-c^2)^{\delta-1} F_\chi^{2,c^2/(1-c^2)}\left(\frac{\| x \|^2+2}{1-c^2}\right)dc,
\end{align*} where $F_\chi^{k,\lambda}$ denotes the cumulative distribution function of a non-central chi-squared random variable with $k$ degrees of freedom and noncentrality parameter $\lambda$, and $f_\chi^{k,\lambda}$ denotes its corresponding density. Furthermore,  the derivative of the probability function is therefore given by
$$
\nabla \Phi(x)=\left(\frac{ \Gamma(2\delta)}{2^{2\delta}\Gamma(\delta)^2}\int_{-1}^1(1-c^2)^{\delta-2} f_\chi^{2,c^2/(1-c^2)}\left(\frac{\| x \|^2+2}{1-c^2}\right)dc\right) x.
$$
On the other hand, 
$$
g(x,(0,c)^\top)\leq -1 \textrm{ for all } x\in \mathbb{R}^n \textrm{ and } c\in [-1,1].
$$
Hence, assumption \eqref{assumption1} is satisfied with $\gamma_0=1/2$. Also, 
$$
|\nabla_x g(x,\xi)|\leq 2\exp\left({\Vert \xi\Vert}\right) \textrm{ for all } x\in \mathbb{R}^n \textrm{ and } \xi \in \mathbb{R}^2,
$$
which implies the growth condition from Definition \ref{def:growthcon}.

The spherical radial decomposition for $\xi$ is $\xi=(0,c)^\top+R\L(c) \Theta$, where $\L(c)=\begin{bmatrix}
    \sqrt{1-c^2} &0 \\
    0& \sqrt{1-c^2}
\end{bmatrix}$ and $R^2\sim \chi_2^2$, i.e., $R$ has the Rayleigh distribution \cite{papoulis2002probability}. The function defined in \eqref{density-like} takes the form $\chi(r)=re^{-r^2/2}$ for $r>0$. Moreover, it is clear that for fixed $c \in \C,v \in \mathbb{S}_1$, we have
$$
g(x,(0,c)^ \top+\sqrt{1-c^2}rv)   \leq 0 \iff r\leq \rho_c(x,v):=\frac{-cv_2 + \sqrt{   \| x \|^2+2 - c^2v_1 ^2    }}{\sqrt{1-c^2}},
$$
which implies that
$$e_c(x,v)= 1-e^{-\rho_c(x,v)^2/2} \quad \textrm{ and }\quad \nabla_x e_c(x,v)= \frac{  x\rho_c(x,v)e^{-\rho_c(x,v)^2/2}}{ \sqrt{ [1-c^2] [ \| x \|^2+2 - c^2v_1 ^2]      } } .$$

The GitHub page \url{https://gmmpfunctions.github.io/GMMpf/} contains simulation results and an R script for conducting further simulations with user-defined parameters. Figure \ref{fig:simu} presents simulation-based results for $x \in [0,4]$, $\delta=2.5$ and a sample size $N=100$, demonstrating that the proposed approximation $\Phi_N$ in \eqref{eq:MCMCprob} is smooth and offers significantly higher accuracy as an estimator of the true probability function compared to the naive estimator \eqref{eq:sampleaverage}, as shown in the leftmost plot. Unlike $\tilde{\Phi}_N$, the smooth profile of $\Phi_N$ admits a continuous derivative $\nabla\Phi_N$, which reasonably approximates the gradient of the probability function $\nabla\Phi$, as illustrated in the plot on the right. 
\begin{figure}[!ht]
    \centering
    \includegraphics[width=.8\linewidth]{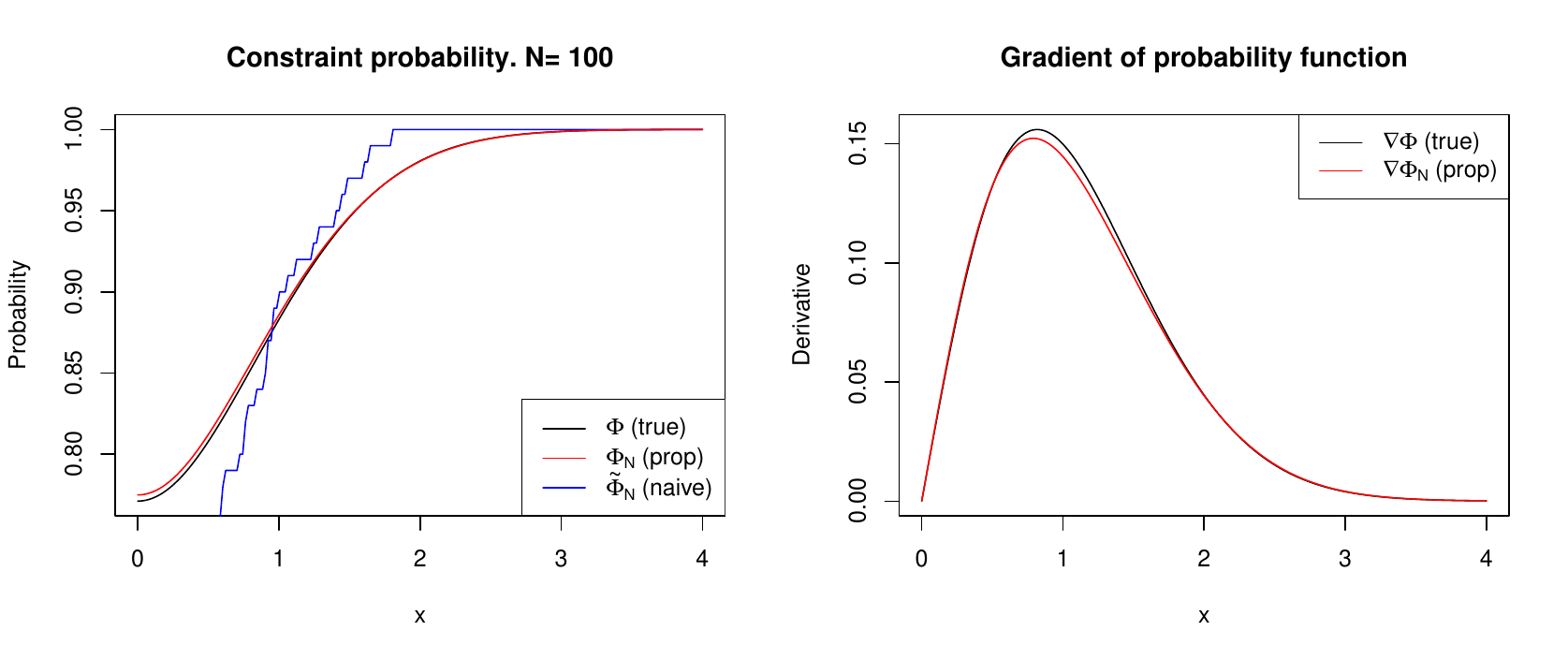}
    \caption{Estimates for $N=100$ and true values of the probability function $\Phi$ (left) and its derivative $\Phi^{\prime}$ (right). Naive estimator \eqref{eq:sampleaverage} in blue, 
 empirical spherical approximation \eqref{eq:MCMCprob} in red, and true functions in black.}
    \label{fig:simu}
\end{figure}

 We estimate the squared error and supremum error loss for the three proposed estimators using $k=1000$ different, for varied sizes of $N$. Table \ref{table:norm2} contains averages and standard deviations of $L^2$ norm (squared loss) and Table \ref{table:norminf} contains the same quantities $L^\infty$ (maximum loss) norms of the estimates $\Phi_N$, $\nabla\Phi_N$ and $\tilde{\Phi}_N$ for different values of $N$. The Github repository contains numerical plots of the error distributions
\begin{table}[!ht]
\centering
\begin{tabular}{|c|c|c|c|}
\cline{1-4}
    \multicolumn{1}{|l|}{$N$}                       &  $\Vert \Phi_N-\Phi\Vert_2^2$                                       & $\Vert \tilde{\Phi}_N-\Phi \Vert^2_2$ &$ \Vert \nabla\Phi_N-\nabla\Phi \Vert_2^2$   \\ \hline
\multicolumn{1}{|l|}{$50$}   &  $(6.8 \pm 9.3)\times 10^{-4} $   &   $(4.3 \pm 4.7)\times 10^{-3} $       &  $(4.3 \pm 5.1)\times 10^{-4} $  \\ \hline
\multicolumn{1}{|l|}{$100$}   &  $(3.4 \pm 4.8)\times 10^{-4} $   &     $(2.1 \pm 2.3)\times 10^{-3} $     &  $(2.2 \pm 2.8)\times 10^{-4} $  \\ \hline
\multicolumn{1}{|l|}{$500$}   &  $(6.4 \pm 9.1)\times 10^{-5} $   & $(4.1 \pm 4.4)\times 10^{-4} $         &  $(4.2 \pm 5.2)\times 10^{-5} $  \\ \hline
\multicolumn{1}{|l|}{$1000$}   &  $(3.2 \pm 4.6)\times 10^{-5} $   & $(2.1 \pm 2.4)\times 10^{-4} $         &   $(2.0 \pm 2.6)\times 10^{-5} $ \\ \hline
\multicolumn{1}{|l|}{$5000$}   & $(6.6 \pm 9.1)\times 10^{-6} $    & $(4.2 \pm 4.5)\times 10^{-5} $         &  $(4.2 \pm 5.2)\times 10^{-6} $  \\ \hline
\multicolumn{1}{|l|}{$10000$}   &   $(3.3 \pm 4.6)\times 10^{-6} $  & $(2.1 \pm 2.2)\times 10^{-5} $         &  $(2.1 \pm 2.7)\times 10^{-6} $  \\ \hline
\end{tabular}
\caption{Average $\pm$ (standard deviation) of squared loss $\Vert \cdot\Vert_2^2$ error of estimates for different values of $N$. First column: empirical spherical approximation $\Phi_N$; second column: naive estimator $\tilde{\Phi}_N$, third column:  empirical spherical approximation $\nabla\Phi_N$.}
\label{table:norm2}
\end{table}

\begin{table}[!ht]
\centering
\begin{tabular}{|c|c|c|c|}
\cline{1-4}
\multicolumn{1}{|l|}{$N$}                           &  $\Vert \Phi_N-\Phi\Vert_{\infty}$                                      & $\Vert \tilde{\Phi}_N-\Phi\Vert_{\infty}$ &$\Vert \nabla\Phi_N-\nabla\Phi\Vert_{\infty}$   \\ \hline
\multicolumn{1}{|l|}{$50$}   & $0.023\pm 0.015$& $0.075\pm 0.031$ & $0.018\pm 0.01$ \\ \hline
\multicolumn{1}{|l|}{$100$}   & $0.016\pm 0.011$ & $0.054\pm 0.021$& $0.013\pm 0.008$  \\ \hline
\multicolumn{1}{|l|}{$500$}   & $0.007\pm 0.005$& $0.023\pm 0.009$& $0.005\pm 0.003$\\ \hline
\multicolumn{1}{|l|}{$1000$}   & $0.004\pm 0.006$ & $0.017\pm 0.006$ & $0.004\pm 0.002$ \\ \hline
\multicolumn{1}{|l|}{$5000$}   & $(2.2\pm 1.3)\times 10^{-3}$ & $(7.5\pm 3.1)\times 10^{-3}$ & $(1.7\pm 1)\times 10^{-3}$ \\ \hline
\multicolumn{1}{|l|}{$10000$}   & $(1.6\pm 1.1)\times 10^{-3}$& $(5.3\pm 2.1)\times 10^{-3}$ & $(1.2\pm 0.7)\times 10^{-3}$ \\ \hline
\end{tabular}
\caption{Average $\pm$ (standard deviation) of supremum loss $\Vert \cdot\Vert_{\infty}$ error of estimates for different values of $N$. First column: empirical spherical approximation $\Phi_N$; second column: naive estimator $\tilde{\Phi}_N$, third column:  empirical spherical approximation $\nabla\Phi_N$.}
\label{table:norminf}
\end{table}

 	
 \section*{Declarations}
\textbf{Data availability} No datasets were generated or analyzed during the current study.\\
\textbf{Conflict of interest} The authors have no conflict of interest to declare.
 
  \section*{Acknowledgments}
  
  We thank the anonymous reviewers for their careful reading of our manuscript and their many insightful comments and suggestions.
 %

\appendix

\section{Well-Posedness of the Probability Function}

\emph{Proof of Proposition \ref{prop:econtinua}}.  Since $\z$ and $\Sigma$ are measurable, the map $(x,v,c,r)\mapsto (x,r\L(c)v+ \z(c))$ is measurable.  Since $g$ is measurable, the composition $(x,v,c,r)\mapsto F(x,v,c,r):=g(x,r\L(c)v+ \z(c))$ is also measurable. Hence, the set $\{(x,v,c,r)\colon r\geq 0, F(x,v,c,r)\leq 0\}$ is measurable. Therefore, the function $(x,v,c,r) \mapsto \chi(r) \1_{ \{ r'\geq 0 :
		g(x, r'\L(c)v+ \z(c))     \leq 0   \}   }(r)$ is measurable. Hence, using Fubini's Theorem we get that \begin{align*}
		(c,v,x) \mapsto  \int\limits_{  \{ r'\geq 0 :
			g(x, r'\L(c)v+ \z(c))     \leq 0   \}      }   \chi( r)dr 
	\end{align*} is measurable, which proves the measurability of $e$.  Now, let us fix $(c,v) \in \C\times \mathbb{S}^{m-1}$ and let $x\in \mathbb{R}^n$ be a $\z$-uniform Slater point. Consider $N >0$ and a sequence $x_k \to x$. Define the functions
\begin{align*}
	\eta_k(r)&:= \chi(r) \1_{ \{ r' \geq  :
		g(x_k, r'\L(c)v+ \z(c))     \leq 0   \}   }(r),  \text{ and } \eta(r):= \chi(r) \1_{ \{ r' \geq  :
		g(x, r'\L(c)v+ \z(c))     \leq 0   \}   }(r) 
\end{align*}
First, note that  $|\eta_k(r)| \leq  \chi(r) $ for all $r\in [0,+\infty)$. Furthermore,  due to the convexity of $g$ with respect to its second argument, it is straightforward to show  that $\eta_k $ converges almost everywhere to $\eta$ as $k \to \infty$ whenever $g(x, r'\L(c)v+ \z(c))     \leq 0$, which holds for every $c\in \C$ except for at most an $\m$-null set as $x$ a $\z$-uniform Slater point. Therefore, by the Lebesgue Dominated Convergence theorem,  we have that 
 $$e_c(x,v)= \int_{\R}  \eta(r)dr =\lim\limits_{ k \to \infty }\int_{\R} \eta_k(r)dr=\lim\limits_{ k \to \infty }e_c(x_k,v),$$ which implies that the function   $e_c(\cdot, v)$ is continuous at $x$.

\section{Differentiability of the probability function}

Define the sets of finite and infinite directions as follows:
\begin{align*}
F(x)&:= \{ (v,c) \in \mathbb{S}^{m-1} \times \C   :   \exists r>0 \text{ s.t. } g(x,\z(c)+r\L(c)v) >0  \},	\\
I(x)&:= \{ (v,c ) \in \mathbb{S}^{m-1} \times \C:   \forall r>0 \;  g(x,\z(c) + r\L(c)v) <0  \}.
	\end{align*}
Furthermore, we introduce the radius function $\rho:\mathbb{R}^n  \times \mathbb{S}^{m-1}\times \C \to \R_+$. defined by
\begin{align}\label{radius_function}
\rho_c(x, v):= \rho(x,v,c):=\sup\{ r \geq 0 : g(x,r\L(c)v + \z(c)) \leq 0 \}.
\end{align}

The following proposition is a direct adaptation of the arguments presented in \cite[Lemmas 3.2, Lemma 3.3 and   3.4]{vanAckooij2022}  (see also \cite{vanAckooij_Henrion_2016}).
\begin{lemma}\label{firstderivative}
	Suppose that condition \eqref{assumption1}  holds at $\bar{x}$. Then, for every $(x,v,c)  \in U \times  \mathbb{S}^{m-1} \times \C$,  the set $\{ r\geq 0 :    g(x,
	r\L(c)v +\z(c)) \leq 0     \} = [0, \rho_c(x,v)]$ (with the convention that $[0, \infty]=[0, \infty)$), and 
 \begin{align}\label{lower_ineq}
     \langle \nabla_z g(x,\z(c)+\rho_c(x,v) \L(c) v), \L(c) v\rangle \geq - \frac{ g(x,\z(c)) }{ \rho_c(x,v)}.
 \end{align}
 Moreover, the function  $e$, defined   in \eqref{function_e}, satisfies  the identity
	\begin{align}\label{repr_e}
		e_c(x,v) =   \displaystyle  \int\limits_{0}^{\rho_c(x,v)} \chi(r)  dr, \quad \text{   for every }(x,v,c) \in \mathbb{B}_\varepsilon(\bar{x}) \times \mathbb{S}^{m-1} \times \C.
	\end{align}
 Furthermore, the function $\rho$ is  $\mathcal{B}(\mathbb{R}^n) \otimes \mathcal{B}(\mathbb{S}^{m-1}) \otimes  \mathcal{F}$ measurable, and for every sequence $(x_k,v_k) \to (x, v)$, we have that $\rho_c(x_k, v_k) \to \rho_c(x,v)$ for all $c\in \C$.  In addition, if $\rho_c(x,v) <+\infty$, then the function $\rho_c$ is continuously differentiable with respect to $x$, with gradient
\begin{equation}\label{formula_gradient_rho}
\nabla_x\rho_{c}(x,v) =
\frac{\nabla_x g(x,\z(c)+\rho_{c}(x,v) \L(c)v) }{\langle \nabla_z g(x,\z(c)+\rho_{c}(x,v) \L(c)v),  \L(c)v\rangle  }.
\end{equation}
\end{lemma}
\begin{proof}
    The equality 
    $$\{ r\geq 0 :    g(x,
	r\L(c)v +\z(c)) \leq 0     \} = [0, \rho_c(x,v)]$$
 follows from the convexity of $g$ with respect to $z$, which directly implies \eqref{repr_e} (see, e.g., \cite[Lemma 2.4]{MR4000225} or \cite[Lemma 3.2]{MR3273343} for further details). Furthermore, the measurability of $\rho$ follows from the equality: for $\nu\geq 0$,
 \begin{align*}
    \{ (x,v,c) \in U \times \mathbb{S}^{m-1} \times \C : \rho(c,x,v)>\nu \} = \{ (x,v,c) \in U \times \mathbb{S}^{m-1} \times \C  :    g(x,
	\nu\L(c)v +\z(c))< 0  \},
 \end{align*}
 where the set on the right-hand side of the equality is measurable due to the continuity of $g$ and measurability of $\z$ and $\L$.
\end{proof}

Now, we establish a lower bound for the inner product between the partial gradient of $g$ with respect to $z$ and the directions over the sphere. This technical inequality follows from the convexity of the function $g$ 
 (with respect to $z$), and its proof is obtained by adapting the arguments from  \cite[Lemma 3.3]{vanAckooij2022} to our setting.

\begin{proposition}
	Suppose that   \eqref{assumption1}  holds, and consider $(x,v,c)  \in U \times  \mathbb{S}^{m-1} \times \C$ with $(v,c) \in F(x)$. Let $r_0$ such that    $\{ r\geq 0 :
	g(x, r\L(c)v+ \z(c))     \leq 0   \}=[0,r_0]$. Then  $\langle \nabla_z g(x,\z(c)+r_0\L(c)v),\L(c)v \rangle\geq -r_0^{-1}g(x,\z(c))>0.$
\end{proposition}

The following result pertains to a technical result of estimating the gradients of the (parameterized) radius function defined in \eqref{radius_function}. Since we cannot assume a priori that the radial probability function is differentiable, the proof of this result relies on techniques from variational analysis and generalized differentiation. We refer to \cite{MR3823783,Rockafellar_Wets_2009} for the primary notation and the tools used in the proof. 

\begin{theorem}\label{theorem01} Let $\bar{x}$ be a $\z$-uniform Slater point, and suppose that $g$ satisfies the exponential growth condition at $\bar{x}$. Then, there exists $\varepsilon>0$ and $\tilde{\C} \subset \C$ with $\m(\C/\tilde{\C})=0$ such that for all $(x,v,c) \in \mathbb{B}_\varepsilon (\bar x)\times \mathbb{S}^{m-1} \times \tilde{\C}$ the function $e_c$ is continuously differentiable with respect to $x$, with gradient
	\begin{align}\label{grad_form}
		\nabla_x e_c(x,v) =\left\{    \begin{array}{cc}
		 \chi(\rho_{c}(x,v))  \nabla_x\rho_{c}(x,v) &\text{ if } (v,c) \in F(x),\\
		 0 &\text{ if } (v,c) \in I(x).
		\end{array}
		\right.
	\end{align}
Furthermore, there exist $\hat{\kappa}\geq 0$ such that 
\begin{align}\label{gra_bound}
	\| \nabla_x e_c(x,v)\| \leq \hat{\kappa}, \quad \text{ for all } (x,v,c) \in \mathbb{B}_\varepsilon(\bar{x}) \times \mathbb{S}^{m-1} \times \C.
\end{align} 
\end{theorem}
\begin{proof}
Let $U$ be the neighborhood in \eqref{assumption1} and $\tilde{C}$ the subset of $\C$ where the inequality \eqref{assumption1} holds. We divide the proof into three claims.\\
 \textbf{Claim 1:}  For all $(x,v,c) \in U\times \mathbb{S}^{m-1} \times \tilde{C}$, the following inclusions hold:
\begin{align}\label{subgrad_form}
 		\hat{\partial}_x e_c( x,v) \subset   \{ 	\chi(\rho_{c}(x,v))  \nabla_x\rho_{c}(x,v) \} \textrm{ for } (v,c)\in F(x) \textrm{ and } \hat{\partial}_x e_c( x,v) \subset \{	0 \} \textrm{ for } (v,c) \in I(x),
 	\end{align} 
  where $\hat{\partial}_x e_c( x,v)$ denotes the Fr\'echet subdifferential of the function $e_c$ with respect to $x$. 
  \\
   \emph{Proof of Claim 1:} 
 Indeed, let $x\in U$. Then, by Lemma \ref{firstderivative}, we have that \eqref{grad_form} holds whenever $(v,c) \in F(x)$.  On the other hand, if $(v,c) \in I(x)$, then $e_c(x,v)=1$, which implies that $e_c$  attains its maximum. It is then straightforward to show that $	\hat{\partial}_x e_c( x,v) \subset \{0\}$.
 
 \textbf{Claim 2:} There exist $\gamma >0$ and $\hat{\kappa}\geq 0$ such that 
 \begin{align*}
 		\hat{\partial}_x e_c( x,v) \subset  \hat{\kappa}\mathbb{B}, \quad \text{for all } (x,v,c) \in \mathbb{B}_\gamma(\bar{x})  \times \mathbb{S}^{m-1} \times \tilde{C}.
 \end{align*} 
    \emph{Proof of Claim 2:} Let $\varepsilon, a, b>0$ be the parameters in the definition of the exponential growth condition. By shrinking $\varepsilon$ if necessary, we can assume that $\mathbb{B}_\varepsilon(\bar x) \subset U$, where $U$ is the neighborhood of $\bar x$ from \eqref{assumption1}. Now, consider $x\in \mathbb{B}_\varepsilon(\bar x)$ and $x^\ast \in \hat{\partial}_x e_c( x,v) $. Then, using \eqref{density-like}, \eqref{formula_gradient_rho}, and \eqref{subgrad_form}, we get   
\begin{equation*}
    \| x^\ast\| \leq 	\chi(\rho_{c}(x,v)) \| \nabla_x\rho_{c}(x,v)  \| =   c_0  \rho_c(x,v)^{m-1} \exp( - \rho(x,v)^2/2 )  \frac{\| \nabla_x g(x,\z(c)+\rho_{c}(x,v) \L(c)v) \| }{|\langle \nabla_z g(x,\z(c)+\rho_{c}(x,v) \L(c)v),  \L(c)v\rangle | } ,
\end{equation*}
where $c_0  = \frac{2\pi^{\frac{m}{2}}	}{\Gamma(\frac{m}{2})}$. Now, as a consequence of \eqref{eqeta0}, \eqref{assumption1}, \eqref{lower_ineq}, and the exponential growth condition, we have that
\begin{equation}\label{ineqxast}
\begin{aligned}
    \| x^\ast\| & \leq  c_0 \gamma_0 \rho_c(x,v)^{m}    \exp( - \rho_c(x,v)^2/2 )  a\exp(b \rho_c(x,v) \|\L(c)\|) \exp( b \| \z(c)\| ) \\ &\leq  a c_0 \gamma_0  \exp( b \eta_0)   \rho_c(x,v)^m \exp(-\rho_c(x,v)^2/2) \exp (b \eta_0  \rho_c(x,v)). 
\end{aligned}  
\end{equation}
Consequently, 
\begin{align*}
        \| x^\ast\| &    \leq  a c_0 \gamma_0  \exp( b \eta_0) \sup\{  r^m \exp(-r^2/2) \exp (b \eta_0 r)  : r\geq 0 \}  =:\hat{\kappa}.
\end{align*}
This concludes the proof of the claim.\\
\textbf{Claim 3:}  For all $(v,c)\in \mathbb{S}^{m-1} \times \tilde{\C}$, the function $x\mapsto e_c(x,v)$ is continuously differentiable over  $\mathbb{B}_{\gamma/2} (\bar x)$.  Consequently,   \eqref{grad_form} and \eqref{gra_bound} hold.\\
   \emph{Proof of Claim 3:} Fix $(v,c) \in \mathbb{S}^{m-1} \times \C$. First, using \cite[Theorem 4.15]{MR3823783}, we conclude that $x\mapsto e_c(x, v)$ is Lipschitz continuous around any $x \in \mathbb{B}_{\varepsilon/2}(\bar x)$. Second, due to   Lemma \ref{firstderivative}, it is enough to prove continuous differentiability when $(v,c) \in I(x)$. Indeed, consider a sequence $x_k \to x$ and $x_k^\ast \in \hat{\partial }_x e_c(x_k,v )$; we will prove that $x_k^\ast \to 0$. If $(v,c) \in I(x_k)$, we have that $\| x_k^\ast\| =0$, so we can assume that $(v,c) \in F(x_k)$ for all $k\in \mathbb{N}$.  Then, by Lemma \ref{firstderivative} and using the estimate given in \eqref{ineqxast}, we have that $x_k^\ast \to 0$. Finally, \cite[Theorem 4.17]{MR3823783} implies that $x\mapsto e_c(x,v )$ is strictly differentiable at $x$ with $\partial_x e_c(x,v) =\{0\}$, and from the arbitrariness of $x$, we obtain its continuous differentiability.  \qed
\end{proof}
\emph{Proof of Theorem \ref{thm:gradient}}
 By defining $\nabla_xe_c(x,v)=0$ whenever $c \notin\tilde{C}$ and observing that this term makes no contribution to the integral, using \eqref{gra_bound} and the classical theorem on the interchange of integration and differentiation, we conclude that the formula \eqref{Formula_gradient_PHI} holds within a suitable neighborhood of $\bar{x}$. Furthermore, considering the integrand in \eqref{Formula_gradient_PHI} and the uniform boundedness in \eqref{gra_bound}, we can establish the local continuity of the gradient around $\bar{x}$. \qed

\section{Numerical Approximation of Probability Functions} 

\emph{Proof of Theorem \ref{thm:limitecontinuo}:} On the one hand, notice that for every $x \in X$, the following inequality holds:
 		\begin{align*}
 			 e_c(x,v) \leq1, \text{ for all } (c,v) \in \C \times \mathcal{S}^{m-1}.
 		\end{align*}
 	On the other hand, by \eqref{gra_bound} and a classical argument of compactness, we can ensure  that there exists $\hat{\kappa}>0$ and an open set $U' \supset X$ such that 
  \begin{align*}
       \|  \nabla_x e_c(x,v)\| \leq \hat{\kappa}, \text{ for all } (x, v , c ) \in U' \times \mathbb{S}^{m-1}\times \C.
  \end{align*}
  For $r_k$ a positive sequence converging to $0$ and arbitrary $\bar{x}  \in X$, define
 \begin{align*}
     \Delta_{1,k}(v,c)=\sup_{x \in \mathbb{B}_{r_k}(\bar{x})}|e_c(x,v)-e_c(\bar{x},v)|, \quad \text{ and } \quad  \Delta_{2,k}(v,c)=\sup_{x \in \mathbb{B}_{r_k}(\bar{x})}|\nabla_xe_c(x,v)-\nabla_x e_c(\bar{x},v)|.
 \end{align*}
  By Theorem \ref{theorem01}, we have that for $j=1,2$, as $k \to \infty$, $\Delta_{j,k}(v,c)\to 0$ $(\nuTheta \otimes \m)$ almost surely. Since $\Delta_{1,k}(v,c) \leq 2$ and $\Delta_{2,k}(v,c) \leq 2\hat{\kappa}$, the dominated convergence theorem implies that for $j=1,2$,
  $$
  \lim_{k\to \infty}\int_{ \C \times \mathbb{S}^{m-1} }\Delta_{j,k}(v,c) d(\nuTheta  \otimes \m)(v,c)=\int_{ \C \times \mathbb{S}^{m-1} }\lim_{k\to \infty}\Delta_{j,k}(v,c) d(\nuTheta  \otimes \m)(v,c)=0.
  $$
  By the Strong Law of Large Numbers \cite[Theorem 6.2]{Bill86} (or equivalently by \cite[Theorem 7.4]{robert_MCMC} when using a MCMC-type sampler), it follows that for fixed $k \in \mathbb{N}$ and $\bar{x} \in X$, as $N \to \infty$, almost surely both
  $$
  \sup_{x :||x-\bar{x}||\leq r_k}|\Phi_N(x)-\Phi_N(\bar{x})|\leq N^{-1}\sum_{i=1}^N \Delta_{1,k}(c_i,v_i) \to \mathbb{E}(\Delta_{1,k}),
  $$
  and 
  $$
  \sup_{x:||x-\bar{x}||\leq r_k}|\nabla\Phi_N(x)-\nabla\Phi_N(\bar{x})|\leq N^{-1}\sum_{i=1}^N \Delta_{2,k}(c_i,v_i) \to \mathbb{E}(\Delta_{2,k}).
  $$
  As the right-hand sides above converge to 0 as $k \to \infty$, for any $\varepsilon>0$, we can find $\tilde{k}$ such that 
  
  $$\sup_{|x-\bar{x}|<r_{\tilde{k}}}|\Phi_N(x)-\Phi_N(\bar{x})|\leq \varepsilon/3 \quad \text{and} \quad\sup_{|x-\bar{x}|<r_{\tilde{k}}}|\nabla\Phi_N(x)-\nabla\Phi_N(\bar{x})|\leq \varepsilon/3.
  $$
Since the sets $\mathbb{B}_{r_{\tilde{k}}}(\bar{x})$ for $\bar{x} \in X$ cover  $X$, the  compactness of $X$ implies that there exists a finite collection $\bar{x}_1, \ldots, \bar{x}_F$ and $K=\max\{\tilde{k}_1,\cdots,\tilde{k}_F\}$ such that $\mathbb{B}_{r_K}(\bar{x}_1), \ldots, \mathbb{B}_{r_K}(\bar{x}_1)$ cover $X$. Therefore, for $\bar{x} \in X$,  there exist indices $1\leq j_0\leq F$ and $1\leq j_1\leq F$ such that 
  $$|\Phi_N(\bar x)-\Phi_N(\bar{x}_{j_0})| \leq \sup_{|x-\bar{x}_{j_0}|<r_K}|\Phi_N(x)-\Phi_N(\bar{x}_{j_0})|\leq \varepsilon/3,$$
  and
  $$|\nabla\Phi_N(\bar x)-\nabla\Phi_N(\bar{x}_{j_1})| \leq  \sup_{|x-\bar{x}_{j_1}|<r_K}|\nabla\Phi_N(x)-\nabla\Phi_N(\bar{x}_{j_1})|\leq \varepsilon/3.$$
 The compactness of $X$ and continuity of $\Phi$ and its gradient ensure that, possibly enlarging $F$ and $K$, for each $\bar{x} \in X$,  there exists $1\leq j\leq F$ such that 
  $$
  \sup_{|x-\bar{x}_j|<r_K}|\Phi(x)-\Phi(\bar{x}_j)|\leq \varepsilon/3 \textrm{ and } \sup_{|x-\bar{x}_j|<r_K}|\nabla\Phi(x)-\nabla\Phi(\bar{x}_j)|\leq \varepsilon/3.
  $$
  For each $1\leq j\leq F$, $h(c,v)=N^{-1}\sum_{i=1}^Ne_{c_i}(\bar{x}_j,v_i)$ and $\tilde{h}(c,v)=N^{-1}\sum_{i=1}^N\nabla e_{c_i}(\bar{x}_j,v_i)$ converges almost surely to $\Phi(\bar{x}_j)$ and $\nabla \Phi(\bar{x}_j)$, respectively. Thus, we can find $\hat{k}$ such that $|\Phi_N(\bar{x}_j)-\Phi(\bar{x}_j)| \leq \varepsilon/3$ and $|\nabla\Phi_N(\bar{x}_j)-\nabla\Phi(\bar{x}_j)| \leq \varepsilon/3$ whenever $N \geq \hat{k}$ for every $j$. Finally, by the triangle inequality, it follows that for arbitrary $\varepsilon>0$, we have for $k>\max \{\hat{k},K\}$,
  $$
  \sup_{x \in X} |\Phi_N(x)-\Phi(x)| \leq \varepsilon \quad \text{and} \quad \sup_{x \in X} |\nabla\Phi_N(x)-\nabla\Phi(x)| \leq \varepsilon.
  $$
  Thus, the uniform convergence of $\Phi_N$ and $\nabla\Phi_N$ is proven.  \qed
\bibliographystyle{plain} 
\bibliography{bibliography}

@article{aktekin2016stochastic,
author = {Aktekin, T. and Ekin, T.},
title = {Stochastic call center staffing with uncertain arrival, service and abandonment rates: A Bayesian perspective},
journal = {Naval Research Logistics (NRL)},
volume = {63},
number = {6},
pages = {460-478},
doi = {https://doi.org/10.1002/nav.21716},
url = {https://onlinelibrary.wiley.com/doi/abs/10.1002/nav.21716},
year = {2016}
}

@book{Bill86,
	author = {Billingsley, P.},
	edition = {Second},
	publisher = {John Wiley and Sons},
	title = {Probability and Measure},
	year = {1986}}

@article{nima2015hydraulic,
author = {Chitsazan, N. and Pham, H. V. and Tsai, F. T.-C.},
title = {Bayesian Chance-Constrained Hydraulic Barrier Design under Geological Structure Uncertainty},
journal = {Groundwater},
volume = {53},
number = {6},
pages = {908-919},
doi = {https://doi.org/10.1111/gwat.12304},
url = {https://ngwa.onlinelibrary.wiley.com/doi/abs/10.1111/gwat.12304},
year = {2015}
}

@article {Faden1985Conditional,
    AUTHOR = {Faden, A. M.},
     TITLE = {The existence of regular conditional probabilities: necessary
              and sufficient conditions},
   JOURNAL = {Ann. Probab.},
    VOLUME = {13},
      YEAR = {1985},
    NUMBER = {1},
     PAGES = {288--298},
      ISSN = {0091-1798}
}

@article{GonzalezDynamic,
author = {González Grandón, T. and Henrion, R. and Pérez-Aros, P.},
title = {Dynamic probabilistic constraints under continuous random distributions},
journal = {Math. Program.},
pages = {1065--1096},
year = {2020},
doi = {10.1007/s10107-020-01593-z}
}

@Article{zbMATH07047588,
 Author = {Hantoute, A. and Henrion, R. and P{\'e}rez-Aros, P.},
 Title = {Subdifferential characterization of probability functions under {Gaussian} distribution},
 Journal = {Math. Program.},
 Volume = {174},
 Number = {1-2 (B)},
 Pages = {167--194},
 Year = {2019},
 Language = {English},
 DOI = {10.1007/s10107-018-1237-9},
 zbMATH = {7047588},
 Zbl = {1421.90098}
}

@article{jaiswal2023Bayesian,
author = {Jaiswal, P. and Honnappa, H. and Rao, V. A.},
title = {Bayesian Joint Chance Constrained Optimization: Approximations and Statistical Consistency},
journal = {SIAM J. Optim.},
volume = {33},
number = {3},
pages = {1968-1995},
year = {2023},
doi = {10.1137/21M1430005},
URL = {  https://doi.org/10.1137/21M1430005},
}

@article{Jones2004MCCLT,
author = {Jones, G. L. },
title = {On the {M}arkov chain central limit theorem},
volume = {1},
journal = {Probability Surveys},
publisher = {Institute of Mathematical Statistics and Bernoulli Society},
pages = {299-320},
year = {2004},
doi = {10.1214/154957804100000051},
URL = {https://doi.org/10.1214/154957804100000051}
}

@InProceedings{pmlr-v108-kirschner20a,
  title = 	 {Distributionally Robust Bayesian Optimization},
  author =       {Kirschner, J. and Bogunovic, I. and Jegelka, S. and Krause, A.},
  booktitle = 	 {Proceedings of the Twenty Third International Conference on Artificial Intelligence and Statistics},
  pages = 	 {2174--2184},
  year = 	 {2020},
  editor = 	 {Chiappa, S. and Calandra, R.},
  volume = 	 {108},
  series = 	 {Proc. Mach. Learn. Res.},
  month = 	 {26--28 Aug},
  publisher =    {PMLR},
  url = 	 {https://proceedings.mlr.press/v108/kirschner20a.html}
}

@article{Lai2011finance,
author = {Lai, T. L. and Xing, H. and Chen, Z.},
title = {Mean–variance portfolio optimization when means and covariances are unknown},
volume = {5},
journal = {Ann. Appl. Stat.},
number = {2A},
pages = {798 -- 823},
year = {2011},
doi = {10.1214/10-AOAS422},
URL = {https://doi.org/10.1214/10-AOAS422}
}

@article{matza2021infinite, 
title={Infinite {G}aussian Mixture Modeling with an Improved Estimation of the Number of Clusters}, 
volume={35},  
number={10}, 
journal={Proceedings of the AAAI Conference on Artificial Intelligence}, 
author={Matza, A. and Bistritz, Y.}, 
year={2021},
pages={8921-8929}
}

@book {mclachlan2004finite,
    AUTHOR = {McLachlan, G. and Peel, D.},
     TITLE = {Finite mixture models},
    SERIES = {Wiley Ser. Probab. Stat.},
 PUBLISHER = {Wiley, New York},
      YEAR = {2000},
     PAGES = {xxii+419},
      ISBN = {0-471-00626-2},
       DOI = {10.1002/0471721182},
       URL = {https://doi.org/10.1002/0471721182}
}

@article{mengersen1999mcmc,
  title={{MCMC} convergence diagnostics: a review},
  author={Mengersen, K. L.},
  journal={Bayesian statistics},
  volume={6},
  pages={415--440},
  year={1999},
  publisher={Oxford University Press}
}

@book {MR3823783,
    AUTHOR = {Mordukhovich, B. S.},
     TITLE = {Variational analysis and applications},
    SERIES = {Springer Monogr. Math.},
 PUBLISHER = {Springer, Cham},
      YEAR = {2018},
     PAGES = {xix+622},
      ISBN = {978-3-319-92773-2; 978-3-319-92775-6},
       DOI = {10.1007/978-3-319-92775-6},
       URL = {https://doi.org/10.1007/978-3-319-92775-6}
}

@book{papoulis2002probability,
  title={Probability, Random Variables, and Stochastic Processes},
  author={Papoulis, A. and Pillai, S.U.},
  series={McGraw-Hill Series in Electrical Engineering. Communications
              and Information Theory},
  publisher={Tata McGraw-Hill},
  year={2002}
}

@book {Prekopa_1995,
    AUTHOR = {Pr\'{e}kopa, A.},
     TITLE = {Stochastic programming},
    SERIES = {Mathematics and its Applications},
    VOLUME = {324},
 PUBLISHER = {Kluwer Academic Publishers Group, Dordrecht},
      YEAR = {1995},
     PAGES = {xviii+599},
      ISBN = {0-7923-3482-5},
       DOI = {10.1007/978-94-017-3087-7},
       URL = {https://doi.org/10.1007/978-94-017-3087-7}
}

@inproceedings{infinitegaussian,
author = {Rasmussen, C.},
year = {2000},
month = {04},
pages = {554-560},
title = {The Infinite {G}aussian Mixture Model},
volume = {12},
 booktitle = {Advances in Neural Information Processing Systems},
 editor = {S. Solla and T. Leen and K. M\"{u}ller},
 publisher = {MIT Press}
}

@Incollection{reynolds2009gaussian,
author="Reynolds, D.",
title="{G}aussian Mixture Models",
bookTitle="Encyclopedia of Biometrics",
editor="Li, S. Z.
and Jain, A. K.",
year="2015",
publisher="Springer US",
address="Boston, MA",
pages="827--832",
isbn="978-1-4899-7488-4",
doi="10.1007/978-1-4899-7488-4_196",
url="https://doi.org/10.1007/978-1-4899-7488-4_196"
}

@book{robert_MCMC,
  author = {Robert, C.P. and Casella, G.},
  publisher = {Springer Verlag},
  title = {Monte {Carlo} statistical methods},
  year = 2004
}

@book{Rockafellar_Wets_2009,
AUTHOR={R.T. Rockafellar and R. J.-B. Wets},
TITLE={Variational Analysis},
NUMBER={},
VOLUME={317},
SERIES={Grundlehren der mathematischen Wissenschaften},
PAGES={734},
DOI={10.1007/978-3-642-02431-3},
PUBLISHER={Springer Verlag Berlin},
EDITION={3rd},
YEAR={2009}
}

@book {Stprog,
    AUTHOR = {Shapiro, A. and Dentcheva, D. and Ruszczy\'{n}ski,
              A.},
     TITLE = {Lectures on stochastic programming},
    SERIES = {MOS-SIAM Ser. Optim.},
    VOLUME = {9},
   EDITION = {Second},
      NOTE = {Modeling and theory},
 PUBLISHER = {SIAM, Philadelphia, PA.},
      YEAR = {2014},
     PAGES = {xviii+494},
      ISBN = {978-1-611973-42-6},
   MRCLASS = {90-01 (90C15)},
  MRNUMBER = {3242164},
}

@article{shapiro2023bayesian,
author = {Shapiro, A. and Zhou, E. and Lin, Y.},
title = {Bayesian Distributionally Robust Optimization},
journal = {SIAM J. Optim.},
volume = {33},
number = {2},
pages = {1279-1304},
year = {2023},
doi = {10.1137/21M1465548},
URL = {  https://doi.org/10.1137/21M1465548}
}

@article {MR3273343,
	AUTHOR = {van Ackooij, W. and Henrion, R.},
	TITLE = {Gradient formulae for nonlinear probabilistic constraints with
	{G}aussian and {G}aussian-like distributions},
	JOURNAL = {SIAM J. Optim.},
	VOLUME = {24},
	YEAR = {2014},
	NUMBER = {4},
	PAGES = {1864--1889},
	ISSN = {1052-6234},
	DOI = {10.1137/130922689},
	URL = {https://doi.org/10.1137/130922689}
}

@article{vanAckooij_Henrion_2016,
	AUTHOR={W. {van Ackooij} and R. Henrion},
	TITLE={({S}ub-) {G}radient formulae for probability functions of random inequality systems under {G}aussian distribution},
	JOURNAL={SIAM J. Uncertain. Quantif.},
	NUMBER={1},
	VOLUME={5},
	PAGES={63-87},
	DOI={10.1137/16M1061308},
	YEAR={2017}
}

@article {MR3881946,
	AUTHOR = {van Ackooij, W. and Aleksovska, I. and Munoz-Zuniga, M.},
	TITLE = {({S}ub-)differentiability of probability functions with
	elliptical distributions},
	JOURNAL = {Set-Valued Var. Anal.},
	VOLUME = {26},
	YEAR = {2018},
	NUMBER = {4},
	PAGES = {887--910},
	ISSN = {1877-0533},
	MRCLASS = {90C15 (49J52 90C30)},
	MRNUMBER = {3881946},
	MRREVIEWER = {Ilya S. Molchanov},
	DOI = {10.1007/s11228-017-0454-3},
	URL = {https://doi.org/10.1007/s11228-017-0454-3}
}

@Article{vanAckooij2022,
 	author={van Ackooij, W.
 	and P{\'e}rez-Aros, P.},
 	title={Generalized Differentiation of Probability Functions: Parameter Dependent Sets Given by Intersections of Convex Sets and Complements of Convex Sets},
 	journal={Appl. Math. Optim.},
 	year={2022},
 	month={Jan},
 	day={29},
 	volume={85},
 	number={1},
 	pages={2},
 	doi={10.1007/s00245-022-09844-5},
 	url={https://doi.org/10.1007/s00245-022-09844-5}
 }

@article{david1948probability,
  title={The probability integral transformation when parameters are estimated from the sample},
  author={David, Florence Nightingale and Johnson, Norman Lloyd},
  journal={Biometrika},
  volume={35},
  number={1/2},
  pages={182--190},
  year={1948},
  publisher={JSTOR}
}

@article {MR4000225,
	AUTHOR = {van Ackooij, W. and P\'{e}rez-Aros, P.},
	TITLE = {Generalized differentiation of probability functions acting on
	an infinite system of constraints},
	JOURNAL = {SIAM J. Optim.},
	FJOURNAL = {SIAM Journal on Optimization},
	VOLUME = {29},
	YEAR = {2019},
	NUMBER = {3},
	PAGES = {2179--2210},
	ISSN = {1052-6234},
	MRCLASS = {90C15},
	MRNUMBER = {4000225},
	MRREVIEWER = {I. M. Stancu-Minasian},
	DOI = {10.1137/18M1181262},
	URL = {https://doi.org/10.1137/18M1181262}
}
\end{document}